\documentclass[a4paper,12pt]{article}
\usepackage{amssymb}
\usepackage{amsthm}
\usepackage{amsmath}
\usepackage{latexsym}
\usepackage{epsfig}
\usepackage{amscd}
\usepackage{graphicx}
\usepackage[dvips]{color}
\newtheorem{theorem}{Theorem}
\newtheorem{lemma}[theorem]{Lemma}

\theoremstyle{definition}

\theoremstyle{remark}

\usepackage[T1]{fontenc}
\usepackage[utf8]{inputenc}

\begin{document}

\title{On a Cyclic Inequality Related to Chebyshev Polynomials}

\date{}

               \author{
  Mohammad Javaheri ~~~~~~~~~~~~~~~~~~~~~~ Harry Shen\\
  \texttt{mjavaheri@siena.edu}~~~~~~~\texttt{hx21shen@siena.edu}\\
  515 Loudon Road \\
  Siena College, School of Science\\
  Loudonville, NY 12211\\
  }

\maketitle

\begin{abstract}
 We show that any weighted geometric mean of Chebyshev polynomials is bounded from above by another Chebyshev polynomial. We also study a related homogeneous cyclic inequality
 $$ \left (\sum_{i=1}^n
    x_i^{(a+b+1)/2} \right )^2
 \geq   \sum_{i=1}^n  x_i    \sum_{i=1}^n x_i^a x_{i+1}^b,$$
where $a,b,x_1,\ldots, x_n$ (with $x_{n+1}=x_1$) are nonnegative. In particular, we prove that the inequality holds when $a=b=1$ and $n\leq 8$ for all nonnegative numbers $x_1,\ldots, x_n$. \footnote{Mathematics Subject Classification (2010): 26D07.\\ Keywords: Chebyshev polynomials, cyclic homogeneous inequality. 
}
\end{abstract}
\section{Introduction}

Chebyshev polynomials of the first kind are defined by the recurrence relation:
\begin{equation}\label{recurr}
T_{n+1}(x)=2xT_n(x)-T_{n-1}(x),
\end{equation}
where $T_0(x)=1$ and $T_1(x)=x$. The most well-known property of Chebyshev polynomials is that they express $\cos(n\theta)$ in terms of $\cos(\theta)$ via the equation $\cos(n\theta)=T_n(\cos \theta)$. Chebyshev polynomials are a special kind of Jacobi polynomials (also known as hypergeometric polynomials), a class of classical orthogonal polynomials. Chebyshev was the first mathematician to have noticed them in 1854, but their importance was not noticed until Hans Hahn rediscovered them and named them after Chebyshev. The polynomials $T_n(x)$ are orthogonal with respect to the inner product 
$$\langle f, g \rangle = \frac{2}{\pi} \int_{-1}^1 f(x)g(x) \dfrac{dx}{\sqrt{1-x^2}}.$$
In other words, $\langle T_m, T_n \rangle =0$ for all positive integers $m\neq n$, and $\langle T_n, T_n \rangle=1$ for all integers $n\geq 0$. 

Chebyshev polynomials have the explicit expression 
\begin{equation}\label{expform}
T_n(x)=\frac{1}{2}  \left (x-\sqrt{x^2-1} \right )^n+ \frac{1}{2}\left (x+\sqrt{x^2-1} \right )^n,
\end{equation}
for $x\geq 1$. By replacing $n$ with $\alpha$ in Equation \eqref{expform}, one can generalize Chebyshev polynomials to functions $T_\alpha: [1,\infty) \rightarrow [1,\infty)$ for all values of $\alpha \in \mathbb{R}$. Equivalently, $T_\alpha$ is defined by
\begin{equation}\label{deft}
T_{\alpha}\left ( \frac{x+x^{-1}}{2} \right )=\frac{x^\alpha+x^{-\alpha}}{2},
\end{equation}
for all $x>0$. From Equation \eqref{recurr}, we can see by induction that $T_{n+1}(x) \geq T_n(x)$ for all $x\geq 1$. One sees that if $|\alpha| \geq |\beta|$, then $T_\alpha(x) > T_\beta(x)$ for all $x >1$ (see Lemma \ref{one}). More generally, any weighted geometric mean of functions of the form $T_\alpha$ is bounded by another function of the same form:

\begin{theorem}\label{chebgen}
Let $t_i$, $1\leq i \leq n$, be nonnegative real numbers such that $\sum_{i=1}^n t_i=1$. If $\alpha^2 \geq \sum_{i=1}^nt_i\alpha_i^2$, then we have
\begin{equation}\label{cheb}
\prod_{i=1}^n(T_{\alpha_i}(x))^{t_i} \leq T_\alpha(x),
\end{equation}
for all $x\geq 1$, where the equality occurs if and only if $x=1$ or $|\alpha|=|\alpha_i|$ for all $1\leq i \leq n$ with $t_i \neq 0$. Conversely, if the inequality \eqref{cheb} holds for all $x\geq 1$, then $\alpha^2 \geq \sum_{i=1}^n t_i \alpha_i^2$.
\end{theorem}

When $n=2$, Theorem \ref{chebgen} can be used to show that if $(c-1)^2 \leq 2ab$, then
$$
(x^c+y^c)^2 \geq (x+y)(x^ay^b+y^ax^b),
$$
for all $x,y\geq 0$, where $c=(a+b+1)/2$. This inequality is a homogeneous cyclic inequality on two variables. We are interested in a general form of this inequality on $n$ arbitrary nonnegative variables:
\begin{equation} \label{stated}
 \left (\sum_{i=1}^n
    x_i^{c} \right )^2
 \geq  \sum_{i=1}^n  x_i    \sum_{i=1}^n x_i^ax_{i+1}^b.
\end{equation}
It turns out that, given fixed values of $a,b \geq 0$, the validity of the cyclic homogeneous inequality \eqref{stated} depends on $n$. One can make the following observations regarding the inequality \eqref{stated}:
\begin{itemize}
\item[i)] If $a+b=1$, then the inequality holds for all $n\geq 1$ (by the Rearrangement inequality \cite{HLP}). 
\item[ii)] Given $a,b\geq 0$ with $a+b\neq 1$, there exists an integer $N(a,b)$ such that the inequality \eqref{stated} holds for all $x_1,\ldots, x_n\geq 0$ if and only if $n\leq N(a,b)$; \cite{J}.  
\item[iii)] Given a positive integer $n$, the set ${\cal O}_n$ of $(a,b) \in [0,\infty)\times [0,\infty)$ for which the inequality \eqref{stated} holds for all $x_1,\ldots, x_n\geq 0$ is a topologically closed and convex subset of $\mathbb{R}^2$.
\end{itemize}

Theorem \ref{case2lem} shows that ${\cal O}_2=\{(a,b): a,b\geq 0~and~(a+b-1)^2 \leq 8ab\}$. The problem of completely characterizing ${\cal O}_n$ for $n>2$ seems difficult, however, a necessary condition is that
$$(a+b-1)^2 \leq 8ab\sin^2(\pi/n).$$
For $n=3$, in Theorem \ref{case3lem}, we will show that 
$$\{(a,b): a,b\geq 0~and~2a+1\geq b\geq (a-1)/2 \geq -b/2\} \subseteq {\cal O}_3.$$
As an example, in section \ref{n11}, we will consider the case of $a=b=1$ and prove the following theorem.
\begin{theorem}\label{main}
Let $x_1,\ldots, x_n$ be nonnegative real numbers. If $n\leq 8$, then
\begin{equation}\label{5gen}
    \left (\sum_{i=1}^n
    x_i^{3} \right )^2
 \geq  \left ( \sum_{i=1}^n  x_i^2\right )  \left ( \sum_{i=1}^n x_i^2x_{i+1}^2\right ),
   \end{equation}
where the equality occurs if and only if $x_1=x_2=\cdots=x_n$. Moreover, the inequality \eqref{5gen} does not hold in general if $n>8$. 
\end{theorem}

\section{Inequalities on Chebyshev polynomials}
In this section, we prove Theorem \ref{chebgen} which gives an upper bound for geometric means of Chebyshev polynomials. We first show that $T_\alpha(x)$ is an increasing function of $|\alpha|$ for a fixed $x>1$. 

\begin{lemma}\label{one}
If $|\alpha| \geq |\beta|$, then $T_\alpha(x) \geq T_\beta(x)$ for all $x\geq 1$. The equality occurs if and only if $x=1$ or $|\alpha|=|\beta|$. 
\end{lemma}

\begin{proof}
Without loss of generality, suppose that $\alpha \geq \beta >0$. By virtue of Equation \ref{deft}, we need to show that the function $\alpha \mapsto x^\alpha+x^{-\alpha}$ is a strictly increasing function of $\alpha > 0$ for a fixed positive $x\neq 1$. The claim then follows from 
$$\frac{d}{d\alpha}(x^\alpha+x^{-\alpha})=(x^\alpha-x^{-\alpha}) \ln x> 0,$$
which holds for all positive $x\neq 1$ and $\alpha > 0$. 
\end{proof}

Next, we prove Theorem \ref{chebgen} in the case of $n=2$. 

\begin{lemma}\label{lemcheb}
If $2\alpha^2 \geq \beta^2+\gamma^2$, then 
\begin{equation}\label{case2cheb}
(T_\alpha(x))^2 \geq T_\beta(x)T_\gamma(x),
\end{equation}
for all $x\geq 1$. 
The equality occurs if and only if $x=1$ or $|\alpha|=|\beta|=|\gamma|$. Conversely, if the inequality \eqref{case2cheb} holds for all $x\geq 1$, then $2\alpha^2 \geq \beta^2+\gamma^2$.
\end{lemma}

\begin{proof}
 We let 
$$G(x)=(x^\alpha+x^{-\alpha})^2 -(x^\beta+x^{-\beta})(x^\gamma+x^{-\gamma} ).$$ 
To prove the inequality \eqref{case2cheb}, it is sufficient to show that $G(x) \geq 0$ for all $x>0$. For $H(x)=xG’(x)$, one has $H(1)=0$ and
\begin{align}\nonumber
xH’(x)&=4\alpha^2(x^{2\alpha}+x^{-2\alpha})-(\beta+\gamma)^2 (x^{\beta+\gamma}+x^{-\beta-\gamma})-(\beta-\gamma)^2(x^{\beta-\gamma}+x^{-\beta+\gamma}) \\ \nonumber
& \geq 4\alpha^2(x^{2\alpha}+x^{-2\alpha})-(\beta+\gamma)^2(x^{2\alpha}+x^{-2\alpha})-(\beta-\gamma)^2(x^{2\alpha}+x^{-2\alpha})\\ \nonumber
&\geq 2(2\alpha^2-\beta^2-\gamma^2)(x^{2\alpha}+x^{-2\alpha}) \geq 0,
\end{align}
since $|2\alpha| \geq |\beta+\gamma|$ and $|2\alpha| \geq |\beta-\gamma|$. It follows that $H’(x) \geq 0$ for all $x>0$. Since $H(1)=0$, we must have $H(x)\geq 0$ for all $x\geq 1$ and $H(x)\leq 0$ for all $0<x\leq 1$. Therefore, $G’(x)\geq 0$ for all $x\geq 1$ and $G’(x)\leq 0$ for all $0<x\leq 1$. Since $G(1)=0$, it follows that $G(x) \geq 0$ for all $x>0$. The equality occurs if and only if $x=1$ or $|2\alpha|=|\beta+\gamma|$ or $|2\alpha|=|\beta-\gamma|$. Therefore, the equality occurs if and only if $x=1$ or $|\alpha|=|\beta|=|\gamma|$. 

For the converse, suppose that $G(x) \geq 0$ for all $x>0$. From the calculations above, we have $G^{\prime}(1)=0$ and $G^{\prime \prime}(1)=4(2\alpha^2-\beta^2-\gamma^2)$. Since $G$ attains a minimum at $x=1$, we must have $G^{\prime \prime}(1) \geq 0$, which implies that $2\alpha^2 \geq \beta^2+\gamma^2$. 
\end{proof}

A function $f(x)$ is said to be {\it concave} on an interval $[a,b]$, if
$$f(tx_1+(1-t)x_2) \geq tf(x_1)+(1-t)f(x_2),$$
for all $x_1,x_2 \in [a,b]$. A function $f(x)$ is said to be {\it midpoint-concave} on an interval $[a,b]$, if 
$$f\left (\frac{x_1+x_2}{2} \right) \geq \frac{f(x_1)+f(x_2)}{2},$$
for all $x_1,x_2 \in [a,b]$. We are now ready to prove Theorem \ref{chebgen}. 
\\
\\
\noindent {\it Proof of Theorem \ref{chebgen}.} Let $f_x: (0,\infty) \rightarrow \mathbb{R}$ be defined by
$$f_x(\alpha)=\ln T_{\sqrt{\alpha}}(x).$$
We show that $f_x$ is a midpoint-concave function of $\alpha>0$ for any fixed value of $x\geq 1$. It follows from Lemma \ref{lemcheb} that 
$$f_x(\alpha)+f_x(\beta)=\ln T_{\sqrt{\alpha}}(x)+\ln T_{\sqrt{\beta}}(x)=\ln T_{\sqrt{\alpha}}(x)T_{\sqrt{\beta}}(x) \leq \ln T_{\sqrt{\gamma}}(x)^2,$$
where $\gamma=(\alpha+\beta)/2$. Therefore, $f_x(\alpha)+f_x(\beta) \leq 2 f_x((\alpha+\beta)/2)$, which means that $f_x$ is a midpoint-concave function on $(0,\infty)$. A theorem of Jensen states that if a function is continuous and midpoint-concave, then it is concave \cite{Jensen,NP}. It follows that $f_x$ is concave. By Jensen’s inequality \cite{HLP}, we conclude that
$$ \sum_{i=1}^n t_i f_x(\alpha_i^2) \leq f_x \left (\sum_{i=1}^n t_i\alpha_i^2 \right ),$$
which implies the inequality \eqref{cheb}.

For the converse, by replacing $x$ with $(x+x^{-1})/2$ in \eqref{cheb} and taking the natural logarithm of both sides, suppose that
$$G(x)=\ln(x^\alpha+x^{-\alpha})-\sum_{i=1}^n t_i \ln(x^{\alpha_i}+x^{-\alpha_i}) \geq 0,$$
for all $x>0$. It is straightforward to show that $G(1)=G^\prime(1)=0$ and 
$$G^{\prime \prime}(1)=\frac{1}{2} \left (\alpha^2 -\sum_{i=1}t_i \alpha_i^2 \right ).$$
Since $G$ attains a minimum at $x=1$, we conclude that $G^{\prime \prime}(1) \geq 0$, and the claim follows. 
\hfill $\square$

\section{A related cyclic homogeneous inequality}
In this section, we study the cyclic homogeneous inequality \eqref{stated}. We first consider the case of $n=2$. 

\begin{theorem}\label{case2lem}
Let $a,b \geq 0$ and $c=(a+b+1)/2$. If $(c-1)^2 \leq 2ab$, then 
\begin{equation}\label{ineq2}
(x^c+y^c)^2 \geq (x+y)(x^ay^b+y^ax^b),
\end{equation}
for all $x,y\geq 0$, and the equality occurs if and only if $x=y$ or $\{a,b\}=\{0,1\}$. 
\end{theorem}

\begin{proof}
With $\alpha=c/2$, $\beta=1/2$, and $\gamma=(a-b)/2$, Lemma \ref{lemcheb} implies that
\begin{equation}\nonumber
((x/y)^{c/2}+(x/y)^{-c/2})^2 \geq ((x/y)^{1/2}+(x/y)^{-1/2})((x/y)^{(a-b)/2}+(x/y)^{(b-a)/2}),
\end{equation}
for all $x\geq 1$, if $2\alpha^2 \geq \beta^2+\gamma^2$ or equivalently $(c-1)^2 \leq 2ab$. The equality occurs if and only if $x/y=1$ or $|\alpha|=|\beta|=|\gamma|$, or equivalently, if and only if $x=y$ or $\{a,b\}=\{0,1\}$. 
\end{proof}

Next, we consider the following general homogeneous cyclic inequality
\begin{equation} \label{gen}
 \left (\sum_{i=1}^n
    x_i^{c} \right )^2
 \geq  \sum_{i=1}^n  x_i    \sum_{i=1}^n x_i^ax_{i+1}^b,
\end{equation}
where $c=(a+b+1)/2$ and $a,b,x_1,\ldots, x_n\geq 0$. One asks that under what conditions on $a,b,c$, the inequality \eqref{gen} holds for all $x_1,\ldots, x_n\geq 0$. It is straightforward to see that if $a+b=1$, then the inequality \eqref{gen} for all $x_1,\ldots, x_n\geq 0$, follows from the Rearrangement inequality \cite[Ch.\ 6]{ZC}. However, if $a+b \neq 1$, then the inequality \eqref{gen} fails to hold if $n$ is large enough \cite{J}. In other words, the validity of the inequality \eqref{gen} for all $x_1,\ldots, x_n\geq 0$ for fixed values of $a,b\geq 0$ with $a+b \neq 1$ depends on $n$. Similarly, given a fixed value of $n$, the inequality \eqref{gen} holds for a specific subset ${\cal O}_n$ of values $(a,b) \in [0,\infty) \times [0,\infty)$. Theorem \ref{case2lem} describes ${\cal O}_2$ completely. For $n>2$, such a complete description seems to be difficult to find. However, in the following theorem, we derive a sufficient condition for the inequality \eqref{gen} in the case of $n=3$.

\begin{theorem}\label{case3lem}
If $2a+1\geq b\geq (a-1)/2 \geq -b/2$, then 
$$(x^c+y^c+z^c)^2 \geq (x+y+z)(x^ay^b+y^az^b+z^ax^b),$$
for all $x,y,z\geq 0$. The equality occurs if and only if $x=y=z$.  
\end{theorem}

\begin{proof}
Without loss of generality, we assume that $a\geq b$. If $b\geq a-1$, then the claim follows from \cite[Prop.\ 2.1]{J}. Thus, suppose that $a-b-1 \geq 0$. Let $x,y,z\geq 0$. By Jensen’s inequality \cite[Ch.\ 7]{ZC}: 
\begin{align} \nonumber
\frac{a+1-b}{2c}  x^{2c}+\frac{b}{c} x^cy^c & \geq x^{a+1}y^b, \\ \nonumber
\frac{b}{2c} x^{2c}+\frac{2b-a+1}{2c}y^{2c}+\frac{a-b}{c} x^cy^c & \geq x^{a}y^{b+1}, \\ \nonumber
\frac{a-b-1}{2c}  x^{2c}+\frac{1}{c} x^cz^{c}+\frac{b}{c} x^cy^c & \geq x^{a}y^bz.
\end{align}
Adding these inequalities yields 
\begin{equation}\label{addxyz}
\frac{2a-b}{2c}x^{2c}+\frac{2b-a+1}{2c}y^{2c}+\frac{a+b}{c}x^cy^c + \frac{1}{c}x^cz^c \geq (x+y+z)x^ay^b.
\end{equation}
Similarly
\begin{align}\label{addxyz2}
\frac{2a-b}{2c}y^{2c}+\frac{2b-a+1}{2c}z^{2c}+\frac{a+b}{c}y^cz^c + \frac{1}{c}x^cy^c \geq (x+y+z)y^az^b, \\ \label{addxyz3}
\frac{2a-b}{2c}z^{2c}+\frac{2b-a+1}{2c}x^{2c}+\frac{a+b}{c}x^cz^c + \frac{1}{c}y^cz^c \geq (x+y+z)z^ax^b.
\end{align}
The claim follows from adding inequalities \eqref{addxyz}-\eqref{addxyz3}.
\end{proof}

\section{The case of $a=b=1$}\label{n11}

In this section, we prove Theorem \ref{main} which states that the inequality \eqref{stated} holds in the case of $a=b=1$ if and only if $n\leq 8$. To show that the inequality \eqref{stated} does not hold for $n\geq 9$, it is sufficient to find a counterexample for the inequality with 9 variables, since if the inequality \eqref{stated} holds for $n$ variables, then it holds for $n-1$ variables. Here is one such counterexample \cite{J}:
$$ x_1=x_9=8.5,~x_2=x_8=9,~x_3=x_7=10,~x_4=x_6=11.5,~x_5=12.$$
It is then left to prove that the inequality \eqref{stated} holds with $a=b=1$ for all $x_1,\ldots, x_8\geq 0$. We first need a lemma. 
\begin{lemma}\label{ineq123}
Let $x_1,\ldots, x_8$ be nonnegative real numbers. Then
$$\sum_{i=1}^8
    x_i^{3} \geq \frac{1}{8} \left ( \sum_{i=1}^8  x_i \right ) \left ( \sum_{i=1}^8  x_i^2 \right ).$$
\end{lemma}

\begin{proof}
By the Power Mean Inequality \cite[Ch.\ III]{Bullen}, one has
$$
\left ( \frac{1}{8}  \sum_{i=1}^8 x_i^3  \right )^{1/3}  \geq \frac{1}{8}\sum_{i=1}^8  x_i ~\mbox{and}~
\left ( \frac{1}{8}  \sum_{i=1}^8  x_i^3 \right )^{1/3}  \geq \left ( \frac{1}{8} \sum_{i=1}^8  x_i^2  \right )^{1/2}.
$$
The claim follows from these inequalities. 
\end{proof}

Now, we are ready to prove Theorem \ref{main}. 
\\
\\
{\it Proof of Theorem \ref{main}.} Equivalently, we show that the maximum value of the function $f: \mathbb{U} \rightarrow \mathbb{R}$ defined by
$$f(x_1,\ldots, x_8)=\dfrac{\sum_{i=1}^8 x_i^4x_{i+1}^4}{\left (\sum_{i=1}^8
    x_i^{6}\right )^2},$$
is 1, where
$$\mathbb{U}= \left \{(x_1,\ldots, x_8): \sum_{i=1}^8 x_i^4=1 \right \}.$$
one has $(\sum_{i=1}^8 x_i^6/8)^{1/6} \geq  (\sum_{i=1}^8 x_i^4/8)^{1/4} $ by the Power Mean Inequality \cite[Ch.\ III]{Bullen}. Therefore, $\sum_{i=1}^8 x_i^6\geq \sqrt{1/8}$ and so the function $f$ is bounded from above on $\mathbb{U}$, hence it attains a positive absolute maximum on the compact set $\mathbb{U}$, say at $(x_1,\ldots, x_8)$. Without loss of generality, we can assume $x_1,\ldots, x_8 \geq 0$. By the method of Lagrange multipliers, there exists a real number $\lambda$ such that 
\begin{equation}\label{lag1}
\frac{1}{A^4} \left (4x_i^3(x_{i-1}^4+x_{i+1}^4)A^2-2AB(6x_i^5) \right )=\lambda (4x_i^3),~\forall i=1,\ldots, 8,
\end{equation}
where $A=x_1^6+\cdots+x_8^6$ and $B=x_1^4x_2^4+\cdots+x_8^4x_1^4$. Therefore,
\begin{equation}\label{lag12}
x_i^4(x_{i-1}^4+x_{i+1}^4)A-3Bx_i^6=\lambda A^3x_i^4,~\forall 1\leq i \leq 8.
\end{equation}

By summing the equations \eqref{lag12}, we have $\lambda=-B/A^2$. We need to show that $A^2 \geq B$. On the contrary, suppose $B>A^2$, and we will derive a contradiction. 

Equations \eqref{lag1} imply that, if $x_i \neq 0$, then
\begin{equation}\label{lag2}
3Bx_i^2=A(x_{i-1}^4+x_{i+1}^4)+AB.
\end{equation}

First, we show that $x_i \neq 0$ for all $i\in \{1,\ldots, 8\}$. On the contrary, and without loss of generality, suppose $x_8=0$ and $x_7 >0$. Given $\epsilon\in (0,x_7)$, let
$$\delta=\delta(\epsilon)=(x_7^4-(x_7-\epsilon)^4)^{1/4},$$
such that $(x_7-\epsilon)^4+\delta^4=x_7^4$, and so $\delta^3 \delta^\prime=(x_7-\epsilon)^3$. We define
 $$F(\epsilon)=f(x_1,\ldots, x_6,x_7-\epsilon, \delta)=\frac{B_\epsilon}{A_\epsilon^2},$$
and compute
$$F^\prime(\epsilon)=\frac{4(x_7-\epsilon)^3}{A_\epsilon^4} \left (A_\epsilon^2(-x_6^4-\delta^4+(x_7-\epsilon)^4+x_1^4)+3A_\epsilon B_\epsilon((x_7-\epsilon)^2-\delta^2) \right ).$$ 
It follows that 
\begin{align}\nonumber
\lim_{\epsilon \rightarrow 0^+}F^\prime(\epsilon) & =\frac{x_7^3}{A^4}(A^2(+x_7^4+x_1^4)-A^2x_6^2+3ABx_7^2) \\ \label{cont0}
& = \frac{x_7^3}{A^4}(A^2(x_1^4+x_7^4)+A^2B)>0,
\end{align}
where we have used equation \eqref{lag2} with $i=7$ to obtain $-A^2x_6^2+3ABx_7^2=A^2B$. 
The inequality \eqref{cont0} is a contradiction with the assumption that $F(\epsilon)$ attains a maximum as $\epsilon \rightarrow 0^+$. We conclude that $x_i>0$ for all $i=1,\ldots, 8$. In particular, equations \eqref{lag2} hold for all $i=1,\ldots, 8$. In the rest of the proof, we let $y_i=x_i^2$. Hence, with $C=B/A$, the equations \eqref{lag2} turn into
\begin{equation}\label{eqlag}
3y_i C =y_{i-1}^2+y_{i+1}^2+B,~\forall 1\leq i \leq 8.
\end{equation}
It follows that 
\begin{align} \nonumber
3(y_i+y_{i+4})C & =2B+ y_{i-1}^2+y_{i+1}^2+y_{i+3}^2+y_{i+5}^2, \\ \nonumber
3(y_{i+2}+y_{i+6})C & =2B+ y_{i+1}^2+y_{i+3}^2+y_{i+5}^2+y_{i+7}^2,
\end{align}
which imply that $y_j+y_{j+4}=y_{j+2}+y_{j+6}$ for all $j$, since $y_{i-1}=y_{i+7}$ as the indices are computed modulo 8. Therefore, there exist nonnegative real numbers $r,s$ such that
\begin{align}\label{eqsub}
y_1+y_5& =y_3+y_7=r, \\ \label{eqsub2}
y_2+y_6 & =y_4+y_8=s.
\end{align}
Equations \eqref{eqlag} imply that
\begin{align}\nonumber
3(y_1-y_3)C & =y_8^2-y_4^2 \\ \nonumber
3(y_2-y_4)C & =y_1^2-y_5^2 \\ \nonumber
3(y_3-y_5)C & =y_2^2-y_6^2 \\ \nonumber
3(y_4-y_6)C & =y_3^2-y_7^2.
\end{align}
Let $\bar y_i=y_i-r/2$ if $i$ is odd, and $\bar y_i=y_i-s/2$ if $i$ is even. It follows that $\bar y_i+ \bar y_{i+4}=0$ for all $i$. Moreover, for $i$ odd, we have
\begin{align}\nonumber
3C(\bar y_i - \bar y_{i+4}) & =3C(y_{i}-y_{i+4})=y_{i-1}^2-y_{i+1}^2+y_{i+3}^2-y_{i+5}^2 \\  \nonumber
 &= (y_{i-1}-y_{i+3})(y_{i-1}+y_{i+3})+(y_{i+1}-y_{i+5})(y_{i+1}+y_{i+5})\\ \nonumber
 &= 2\bar y_{i-1}s+2 \bar y_{i+1}s,
\end{align}
which implies that $\bar y_i=(\bar y_{i-1}+\bar y_{i+1})s/(3C)$ for odd $i$. Similarly, $\bar y_i=(\bar y_{i-1}+\bar y_{i+1})r/(3C)$ for even $i$. It then follows that
$$\bar y_i=\frac{s}{3C}(\bar y_{i-1}+\bar y_{i+1})=\frac{rs}{9C^2}(\bar y_{i-2}+\bar y_{i}+\bar y_{i}+\bar y_{i+2})=\frac{2rs}{9C^2}\bar y_i,$$
for odd $i$, and similarly for even $i$. We claim that $9C^2 \neq 2rs$. On the contrary, suppose $9C^2 = 2rs$, and so, since $C=B/A>A$, we must have
\begin{equation}\label{cont}
6 \sqrt{2}A < 6\sqrt{2} C \leq 4\sqrt{rs} \leq 2(r+s) \leq \sum_{i=1}^8 y_i.
\end{equation}
However, by Lemma \ref{ineq123}, we have $8A \geq \sum_{i=1}^8 y_i$ which contradicts \eqref{cont}, since $6\sqrt{2} >8$. Thus $9C^2 \neq 2rs$, and so $\bar y_i=0$ for all $i$. Therefore, $y_1=y_3=y_5=y_7=r/2$, and $y_2=y_4=y_6=y_8=s/2$. So we have $r^2+s^2=4((r/2)^2+(s/2)^2)=\sum_{i=1}^8 y_i^2=1$ and
\begin{equation}\label{fx}
f(x_1,\ldots, x_8)=\dfrac{8(r/2)^2(s/2)^2}{(4(r/2)^3+4(s/2)^3)^2} = \dfrac{2r^2s^2}{(r^3+s^3)^2}\leq 1,
\end{equation}
since it follows from $r^2+s^2=1$ that
$$r^3+s^3 \geq 2 \left ( \frac{r^2+s^2}{2} \right )^{3/2} \geq 2 \left (\frac{1}{2} \right )^{3/2}\geq \frac{r^2+s^2}{\sqrt 2} \geq \sqrt 2 rs.$$
The equality occurs in \eqref{fx} if and only if $r=s$ (and so $y_1=y_2=\cdots=y_8$); hence, the equality in \eqref{5gen} occurs if and only if $x_1=x_2=\cdots=x_8$. \hfill $\square$

%
%

\end{document}